\documentclass[a4paper,10pt]{article}

\usepackage[utf8]{inputenc}
\usepackage{amsmath, amsfonts, amsthm}
\usepackage[usenames,dvipsnames]{color}
\usepackage[normalem]{ulem}
\usepackage{subfig}
\usepackage{placeins}

\usepackage[]{graphicx}
\newtheorem{theorem}{Theorem}
\newtheorem{assumption}{Assumption}

\newtheorem{lemma}[theorem]{Lemma}

\newcommand{\eps}{\varepsilon}
\renewcommand{\epsilon}{\eps}
\begin{document}

\title{Numerical solution of stochastic partial differential equations with correlated noise}
\author{
Minoo Kamrani\footnote{Department of Mathematics, Faculty of Sciences,
Razi University, Kermanshah, Iran}
\ \&\ 
Dirk Bl\"omker\footnote{Institut f\"ur Mathematik, Universit\"at Augsburg, 86159 Augsburg, Germany}
}

\newcommand{\R}{{\mathbb R}}
\newcommand{\CO}{{\cal O}}
\maketitle

\begin{abstract}
In this paper we investigate the numerical solution 
of stochastic partial differential equations (SPDEs) for a
 wider class of stochastic equations. We focus on non-diagonal colored noise instead of 
the usual space-time white noise. By applying a spectral Galerkin method 
for spatial discretization and a numerical scheme in time introduced by Jentzen \& Kloeden, 
we obtain the rate of  path-wise convergence in the uniform topology. 
The main assumptions are either uniform bounds on the spectral Galerkin approximation
or uniform bounds on the numerical data.
Numerical examples illustrate the theoretically predicted convergence rate.
\end{abstract}
 \\\hrule

  

\section{Introduction}

Let $T>0$,  
$(\Omega,\mathcal{F},\mathbb{P})$ be a probability space and $V$ is a Banach space. 
Suppose the space-time continuous stochastic process 
$X:[0,T]\times \Omega \rightarrow V$ 
is the unique solution of the the following 
stochastic partial differential equation (SPDE) 
\begin{equation}\label{intr1}
dX_{t}=\left[AX_{t}+F(X_{t})\right]dt+dW_{t},
\qquad X_{t}(0)=X_{t}(1)=0,
\quad X_0=0,
\end{equation}
for $t\in [0,T]$ and $x\in (0,1)$, where the operator $A$ denotes an unbounded operator, 
for example the Laplacian. The noise is given 
by a Wiener process $W_{t}$, $t\in [0,T]$ defined later.

The main purpose of this article is to consider a spectral Galerkin 
approximation of (\ref{intr1}) in $L^\infty$, where the noise is colored.
The main results are formulated in an abstract way so that 
in principle they should
apply to other approximation methods like finite elements, but
here we only verify the applicability for spectral Galerkin methods
for simplicity of presentation.

A key point is the uniform bound on the numerical data. Alternatively,
we can uniformly bound the Galerkin approximation,
which is for spectral methods frequently straightforward to verify, by 
using energy-type a-priori estimates.

Of course the result should apply for higher dimensional 
domains, differential operators of higher order, 
or other boundary conditions like Dirichlet,
but as an example we stick with this relatively simple situation here. 

In \cite{Ref1362} the Galerkin approximation was already considered 
for a stochastic Burgers equation with colored noise, 
but here we present this method in a more general setting, 
and not only for the Burgers equation. 
The main novelty, as in  \cite{Ref1362} or \cite{Ref1},
is to bound the spatial and temporal discretization error
in the uniform topology. The space of continuous or H\"older-continuous 
functions is a natural space for stochastic convolutions.
For instance, if for space-time white noise the stochastic convolution 
is in $L^2$ in space, it is already continuous. 
In a recent publication  \cite{Ref2014} Cox \& van Neerven established 
a time-discretization error in H\"older spaces, but the spatial 
error in UMD-spaces. We strongly believe, that 
working in fractional Sobolev-spaces $W^{\alpha,p}$ 
with $\alpha>0$ small and $p\gg1$ large
should yield similar results than ours,
but we present here a simple proof yielding uniform bounds in time only. 

In \cite{Ref1, Ref2} the Galerkin approximation was considered for 
a simple case of SPDEs of the type of (\ref{intr1}), 
either without time-discretization or in different spaces. Moreover,
the Brownian motions in the Fourier expansion of the noise are independent. 
But in general the spatial 
covariance operator of the forcing does not necessarily 
commute with the linear operator $A$, thus 
we consider here the case where the Brownian 
motions are not independent.

Many authors have investigated the spectral Galerkin method for 
this kind of equation with space-time white noise. 
See for example \cite{Ref13, Ref1303, Ref2009, Ref2, Ref2310, Ref1312, Ref2304}.
There are also many articles about finite difference methods 
\cite{Ref5, Ref4, Ref13, Ref1308, Ref16}. 
The existence and uniqueness of solutions of the stochastic equation
was studied in \cite{Ref1995, Ref6} for space-time white noise.
In our proofs, as the nonlinearity allows for polynomial growth,
we do not rely on the global existence of solutions,
but assume that the numerical approximation remains uniformly bounded.
In the limit of fine discretization, this will ensure global existence 
of the solutions and a global error bound for the numerical approximation.

Our aim here is to extend the results of \cite{Ref1362} 
to the case of more general nonlinearities, with local 
Lipschitz conditions and polynomial growth. 
For spatial discretization of equation (\ref{intr1}) 
we apply a spectral Galerkin approximation as already discussed in \cite{Ref1} 
and for the time discretization we follow the method proposed in \cite{Ref2}. 

It should be mentioned that the spatial discretization error is 
obtained by the results of \cite{Ref1, Ref1362}. 
We will recall their main results in Section 1. 
In this article we focus on the time discretization. 
Not treated in \cite{Ref1} but already in \cite{Ref1362}, 
we consider here also the case of colored noise being not diagonal 
with respect to the eigenfunctions 
of the Laplacian. As the final  result  we obtain an error estimate for 
the full space-time discretization.
The main result of \cite{Ref1362} in combination with the results  
presented in this paper yields the
convergence results  in the uniform topology of continuous 
functions for the numerical approximations 
of a wider class of SPDEs with colored noise.
The key assumption is a uniform bound on the 
numerical approximations, that allows for local Lipschitz-conditions only.

The paper is organized as follows.
Section 2 gives the setting and the assumptions. 
In Section 3 we recall the results on the  spatial discretization error, 
and in Section 4 estimates for the temporal error are derived. 
Finally, in the last section a simple numerical example is presented,
in order to illustrate the results.


 \section{Setting and assumptions}


Let $V,W$ be two $\mathbb{R}$-Banach spaces such that $V\subseteq W$. 
Suppose that the unbounded and invertible linear operator $A$  
generates an analytic semigroup $S_{t}$ on $V$ 
that extends to the larger space $W$ , i.e., $S_{t}:W \rightarrow V$. 
Especially, $S_{t+s}=S_tS_s$ and $S_0=Id$.

Consider the following equation
\begin{equation}\label{intr}
dX_{t}=\left[AX_{t}+F(X_{t})\right]dt+dW_{t},
\qquad X_{t}(0)=X_{t}(1)=0,
\quad X_0=0,
\end{equation}
for $t\in [0,T]$ and $x\in (0,1)$. It should be mentioned that just for simplicity we assumed these initial and boundary conditions. \\
Suppose there are bounded linear operators
$P_{N} :V\rightarrow V$. The example we have 
in mind is  the spectral Galerkin method given by the orthogonal projection
$P_{N}v=\sum_{i=1}^{N}\int_{0}^{1}e_{i}(s) v(s) ds \cdot e_{i}$,
where $\{e_i\}_{i\in\mathbb{N}}$ are an orthonormal 
basis of eigenfunctions of $A$.
But any other approximation method like finite elements 
should work in a similar way, if we can satisfy our assumptions 
for the projections.

Consider the following assumptions already made in \cite{Ref1}.


\begin{assumption}[Semigroup]
Suppose for the semigroup, that $S:(0,T]\rightarrow L(W,V)$ is 
a continuous mapping that commutes with $P_{N}$ and satisfies
for given constants  $\alpha, \theta \in [0,1)$ and $\gamma \in (0,\infty)$
\begin{equation}\label{S1}
\sup_{0<t\leq T}\left(t^{\alpha}\|S_{t}\|_{L(W,V)}\right)< \infty,
~~\sup_{N\in \mathbb{N}} \sup_{0\leq t\leq T}(t^{\alpha}N^{\gamma}\|S_{t}-P_{N}S_{t}\|_{L(W,V)})< \infty,
\end{equation}
and 
\begin{equation}
\label{e:SGtime}
\|A^{\theta} S_{t}\|_{L(V,V)} \leq C t^{-\theta}
\quad\text{together with}
\quad
\|A^{-\theta}(S_{t}-I)\|_{L(V,V)} \leq t^{\theta}.
\end{equation}
\end{assumption}

The first assumption is crucial for  the spatial discretization,
while the second assumption (\ref{e:SGtime}) is mainly needed 
for the result on time-discretization,
in order to bound differences of the semigroup.
For example, for analytic semigroups generated by the Laplacian,
this is usually straightforward to verify. See for example \cite{Ref2013}.


\begin{assumption}[Nonlinearity]
\label{ass:nonlin}
Let $F:V\rightarrow W$ be a continuous mapping, 
which satisfies the following local Lipschitz condition.
There is a nonnegative integer $p$ and a constant   $L>0$ such that for all $u,v \in V$
\begin{equation}\label{drift condi}
\|F(u)-F(v)\|_{W}\leq L\|u-v\|_{V}(1+\|u\|_{V}^{p}+\|v\|_{V}^{p})\;.
\end{equation}

\end{assumption}

Let us remark that it is not a major restriction that we
assumed the operator $A$ to be invertible, as we can always 
consider for some constant $c$ the operator $\tilde A=A+cI$ and the nonlinearity
$\tilde F=F-cI$.
 

\subsection{The Ornstein-Uhlenbeck process}
\label{suse:OU}


\begin{assumption}[Ornstein-Uhlenbeck]
Let $O:[0,T]\times \Omega\rightarrow V$ be a stochastic process with continuous sample paths and there exists some $\gamma\in (0,\infty)$ such that
\begin{equation}\label{S3}
\sup_{N\in \mathbb{N}} \sup_{0\leq t\leq T} N^{\gamma}\|O_{t}(\omega)-P_{N}(O_{t}(\omega))\|_{V}< \infty,
\end{equation}
for every $\omega \in \Omega$.\\
Moreover,
\begin{equation}\label{eq45}
\sup_{0\leq t_1\leq t_2\leq T}\frac{\big\|O_{t_2}(\omega)-O_{t_1}(\omega)\big\|_V}{(t_2-t_1)^{\theta}}<\infty,\\
\end{equation}
for some $\theta \in (0,\frac{1}{2})$.
\end{assumption}


In order to give an example for this assumption we focus for the remainder of this subsection now on $L^2[0,1]$
with basis functions $e_k$ are given by the standard Dirichlet basis,
where for every  $k\in \mathbb{N}$ 
 \begin{equation*}\label{eq2}
e_{k}:[0,1]\rightarrow \mathbb{R},~~~ e_{k}(x)=\sqrt{2}\sin(k\pi x),~~x\in [0,1],
\end{equation*}
are smooth functions. For every $k \in N$ define the real numbers $\lambda_{k}=(\pi k)^{2}\in \mathbb{R}$.

Furthermore, let $Q$ be a symmetric non-negative operator,
given by the convolution with a translation invariant positive definite  kernel $q$.  
This means
\begin{equation}\label{eq:defQ}
<Q e_{k},e_{l}>=\int_0^{1}\int_0^{1}e_{k}(x)e_{l}(y)q(x-y)dy dx,
\end{equation}
for $k,l\in \mathbb{N}$.
Note that $Q$ is diagonal with respect to the standard Fourier basis,
but in general not with the Dirichlet basis.

We think of $Q$ being the covariance operator of a Wiener process $W$ in $L^2(0,1)$
and $q$ being the spatial correlation function of the noise process $\partial_t W(t)$.
See for example \cite{DB:noise} for a detailed discussion.

Let $\beta^{i}:[0,T]\times \Omega\rightarrow \mathbb{R},i\in \mathbb{N},$ 
be a family of Brownian motions that are not necessarily independent.
We usually think of $\beta^{i}(t)=\langle W(t),e_i\rangle$. See the discussion at end of this subsection.

Note that the variance of the  Brownian motion is $\sigma_k^2 = <Q e_{k},e_{k}>$,
which means that for $\sigma_k\not=0$ the process $\sigma_i^{-1} \beta^{i}(t)$
is a standard Brownian motion. Moreover, the $\beta^{i}$'s 
are correlated as given by
\begin{equation*}\label{eqn2}
\mathbb{E}\left[\beta^{k}(t)\beta^{l}(t)\right]= \langle Q e_{k},e_{l} \rangle \cdot t,~~ \text{for } k,l\in \mathbb{N}.
\end{equation*}
For the regularity assume that for some $\rho>0$ we have 
\begin{equation}\label{MatQ}
\sum_{i\in \mathbb{N}}\sum_{j\in \mathbb{N}}\|i\|_2^{\rho-1}\|j\|_2^{\rho-1}|\langle Qe_i,e_j\rangle |< \infty.
\end{equation}
This is for a diagonal operator $Q$ a condition on the trace of $\Delta^{\rho-1}Q$
being finite.

Using (\ref{MatQ}) together with Lemma 4 in \cite{Ref1362}, there exists 
a stochastic process $O:[0,T]\times \Omega\rightarrow V$,
which is the Ornstein-Uhlenbeck process (or stochastic convolution) 
given by the semigrup generated by the Dirichlet Laplacian 
and the Wiener process $W(t)=\sum_{k\in\mathbb{N}} \beta^{k}(t)e_k$.
Furthermore,  Lemma 4 in \cite{Ref1362} assures that $O$ satisfies    
Assumption 3, for all  $\theta \in (0,\min\{\frac{1}{2},\frac{\rho}{2}\})$ and $\gamma \in (0,\rho)$ 
with
\begin{equation}\label{eq46}
\mathbb{P}\Big[\lim_{N\rightarrow \infty}\sup_{0\leq t\leq T}
\Big\|O_{t}- \sum_{i=1}^N
\Big(-\lambda_i\int_0^te^{-\lambda_i(t-s)}\beta^{i}(s)ds+\beta^{i}(t)\Big)e_i\Big\|_{C^0([0,1)]}=0\Big]=1.
\end{equation} 
Let us comment a little bit more on the $Q$-Wiener process.
As $Q$ is a symmetric Hilbert-Schmidt operator, 
there exists an orthonormal basis
$f_k$ given by eigenfunctions of $Q$ with $\alpha_k^2f_k=Qf_k$. 
Using standard theory of \cite{Ref1995}, there is
a family of i.i.d.\ Brownian motions $\{B_k\}_{k\in\mathbb{N}}$ such that
$W(t)=\sum_{ k\in\mathbb{N}}\alpha_k B_k(t) f_k \in L^2([0,1])$. 
We can then define 
\[\beta_k(t)=\langle W(t),e_k\rangle_{L^2} 
=  \sum_{ \ell\in\mathbb{N}}\alpha_\ell B_\ell(t) \langle f_\ell,e_k\rangle_{L^2}.
\]


\subsection{Bounds and solutions}


Let us first assume boundedness of the spectral Galerkin approximation.
This will assure the existence of mild solutions later on.
We will discuss later how to relax this condition to boundedness of the 
numerical data alone.

\begin{assumption}
Let $X^{N}:[0,T]\times \Omega\rightarrow V,~~N\in \mathbb{N}$, 
be a sequence of stochastic processes with continuous sample paths such that
\begin{equation}
\label{e:Galbou}
\sup_{M\in \mathbb{N}}\sup_{0\leq s\leq T}\|X_s^{M}(\omega)\|_{V}<\infty
\end{equation}
and
\begin{equation}\label{S4}
X_{t}^{N}(\omega)=\int_0^{t}P_{N}S_{t-s}F(X_s^{N}(\omega))ds+P_{N}(O_{t}(\omega)),
\end{equation}
for every $t\in [0,T], \omega\in \Omega$ and every $N\in \mathbb{N}$.
\end{assumption}


From \cite{Ref1} we have the following theorem about existence of solutions.


\begin{theorem}\label{theoremBJ}
Let Assumptions 1-4 be fulfilled. Then, there exists a unique
stochastic process $X : [0, T ]\times \Omega\rightarrow V$ with continuous sample paths, which satisfies
\begin{equation}\label{eqBJ1}
X_{t}(\omega)=\int_0^{t}S_{t-s}F(X_s(\omega))ds+O_{t}(\omega),
\end{equation}
for every $t \in [0, T ]$ and every $\omega \in \Omega$. Moreover, there exists a $\mathcal{F}/\mathcal{B}([0,\infty))$-measurable
mapping $C:[0,\infty)\rightarrow \Omega$ such that
\begin{equation}\label{eqBJ2}
\sup_{0\leq t \leq T} \parallel X_{t}(\omega)-X_{t}^{N}(\omega)\parallel_{V} \leq C(\omega)\cdot N^{-\gamma},
\end{equation}
holds for every $N \in \mathbb{N}$ and every $\omega \in \Omega$, 
where $\gamma \in (0,\infty)$
is given in Assumption 1 and Assumption 3.
\end{theorem}



\section{Time discretization}
%

%
For the time discretization of the finite dimensional SDE
(\ref{S4}) we follow the method proposed in \cite{Ref2},
which was also used in \cite{Ref1362}. 
Fix a small time-step $\Delta t>0$  and 
define the discrete points via the mapping 
$Y_{m}^{N,M}:\Omega\rightarrow V$ for $m\in \{1,...,M\}$ by
\begin{equation}\label{Teq1}
Y_{m+1}^{N,M}(\omega)=S_{\Delta t}\Big(Y_{m}^{N,M}(\omega)
+\Delta t (P_{N}F)(Y_{m}^{N,M}(\omega))\Big)
+P_{N}\Big(O_{(m+1)\Delta t}(\omega)-S_{\Delta t}O_{m\Delta t}(\omega)\Big).
\end{equation}
Thus $Y_{m}^{N,M}$, $m\in \{1,...,M\}$ should be the approximation of the 
spectral Galerkin approximation $X^N$ (see (\ref{Teqq2}) below) 
at times $m \cdot (\Delta t)$.

For simplicity of presentation, we first assume in addition to
(\ref{e:Galbou})  that our numerical data is uniformly bounded:

\begin{assumption}
 For the numerical scheme (\ref{Teq1}) we assume
\begin{equation}\label{Numeric data}
\sup_{0\leq m\leq M}\sup_{N,M\in \mathbb{N}}\|Y_{m}^{N,M}\|_{V}< \infty.
\end{equation}
\end{assumption}
%
%
%
Therefore, in all the examples that one wants to study, we need to verify that 
both bounds (\ref{Numeric data}) and (\ref{Teq1}) are true,
which might be quite involved. We will comment later 
on the extension of the approximation result, in case  either (\ref{Numeric data}) or (\ref{Teq1}) is not verified.

Our aim is now to obtain the discretization error in time
\begin{equation}\label{Teq2}
\|X_{m\Delta t}^{N}(\omega)-Y_{m}^{N,M}(\omega)\|_{V},
\end{equation}
where
\begin{equation}\label{Teqq2}
X_{m\Delta t}^{N}(\omega)=\int_0^{m\Delta t}P_{N}S_{m\Delta t-s}F(X_s^{N}(\omega))ds+O_{m\Delta t}^{N}(\omega)
\end{equation}
is the solution of the spatial discretization, which is evaluated at the grid points.
It should be mentioned that for simplicity of notation, 
during this section $C(\omega,\alpha,\theta)>0$ 
is a random constant which changes from line to line.


\begin{lemma}\label{lemma8}
Suppose Assumptions 1-4 are true.
Let $X^{N}:[0,T]\times \Omega\rightarrow V$ be the unique 
adapted stochastic process with continuous sample paths 
in (\ref{S4}) and  $O^{N}:[0, T]\times \Omega \rightarrow V$ 
is the stochastic process defined in Assumption 4 in (\ref{eq46}). 
Then for all  $\vartheta \in (0,1-\alpha)$, then there is a random  variable $C:\Omega\rightarrow [0,\infty)$ such that 

\begin{equation*}\label{Teq3}
\begin{split}
&\Big\|(X_{t_2}^{N}(\omega)-O_{t_2}^{N}(\omega))-(X_{t_1}^{N}(\omega)-O_{t_1}^{N}(\omega))\Big\|_{V} 
\leq C(\omega)(t_2-t_1)^{\vartheta},
\end{split}
\end{equation*}
for every $N\in\mathbb{N}$,  $\omega\in \Omega$ 
 and all $t_1,t_2\in [0,T]$, with $t_1<t_2$. 
\end{lemma}
Note that $\alpha$ was introduced in Assumption 1.

\begin{proof}
\begin{equation}
\begin{split}
&\Big\|(X_{t_2}^{N}(\omega)-O_{t_2}^{N}(\omega))-(X_{t_1}^{N}(\omega)-O_{t_1}^{N}(\omega))\Big\|_{V} \\& =
\Big\|\int_{t_{1}}^{t_{2}}P_{N}S_{t_{2}-s}F(X_{s}^{N}(\omega)) ds+\int_{0}^{t_{1}}P_{N}(S_{t_{2}-s}-S_{t_{1}-s})F(X_{s}^{N}(\omega))ds\Big\|_{V}
 \\& \leq \int_{t_{1}}^{t_{2}}\|P_{N}S_{t_{2}-s}\|_{L(W,V)}.\|F(X_{s}^{N}(\omega)\|_{W}ds+\int_{0}^{t_{1}}\|P_{N}(S_{t_{2}-s}-S_{t_{1}-s})\|_{L(W,V)}\|F(X_{s}^{N}(\omega))\|_{W} ds
\\& \leq \sup_{0\leq s\leq T} \|F(X_{s}^{N}(\omega))\|_{W} \int_{t_{1}}^{t_{2}}(t_{2}-s)^{-\alpha}ds +\int_{0}^{t_{1}}\|P_{N}S_{t_{1}-s}(S_{t_{2}-t_{1}}-I)\|_{L(W,V)}\|F(X_{s}^{N}(\omega))\|_{W}ds 
\\& \leq \Big(\int_{t_{1}}^{t_{2}}(t_{2}-s)^{-\alpha}ds+\int_{0}^{t_{1}}\|P_{N}S_{t_{1}-s}A^{\vartheta}\|_{L(W,V)}\|A^{-\vartheta}(S_{t_{2}-t_{1}}-I)\|_{L(V,V)}ds\Big)\sup_{0\leq s\leq T} \|F(X_{s}^{N}(\omega))\|_{W} 
\\& \leq C(\omega) \Big(\int_{t_{1}}^{t_{2}}(t_{2}-s)^{-\alpha}ds+\int_{0}^{t_{1}}\|P_{N}S_{\frac{t_{1}-s}{2}}\|_{L(W,V)}\|S_{\frac{t_{1}-s}{2}}A^{\vartheta}\|_{L(V,V)}\|A^{-\vartheta}(S_{t_{2}-t_{1}}-I)\|_{L(V,V)}ds\Big)
\\& \leq C(\omega)\Big( (t_{2}-t_{1})^{1-\alpha}+\int_{0}^{t_{1}}(t_{1}-s)^{-\alpha-\vartheta}ds\cdot(t_{2}-t_{1})^{\vartheta}\Big)
\\&
  \leq C(\omega)\Big((t_{2}-t_{1})^{1-\alpha}+(t_{2}-t_{1})^{\vartheta}\Big)
 \leq C(\omega) (t_{2}-t_{1})^{\vartheta},
 \end{split}
\end{equation}
where we have used (\ref{e:SGtime}).
\end{proof}


Before we start to bound the first part of the error, we
define
\begin{equation*}\label{Teq8}
\begin{split}
R(\omega):=
&\sup_{N\in \mathbb{N}}\sup_{0\leq s\leq T}\|F(X_s^{N}(\omega))\|_{W}
+\sup_{N\in \mathbb{N}}\sup_{0\leq s\leq T}\|X_s^{N}(\omega)\|_{V}
\\&
+\sup_{0\leq t_1,t_2\leq T}\|O_{t_2}(\omega)-O_{t_1}(\omega)\|_{V}|t_2-t_1|^{-\vartheta}
\\&
+\sup_{N\in \mathbb{N}}\sup_{0\leq t_1,t_2\leq T}
\|X_{t_2}^{N}(\omega)-O_{t_2}^{N}(\omega)-(X_{t_1}^{N}(\omega)-O_{t_1}^{N}(\omega))\|_{V}|t_2-t_1|^{-\vartheta}.
\end{split}
\end{equation*}
provided  $\vartheta \in (0,\min\{\theta,1-\alpha\})$. 

From  Assumption 2, 4, (\ref{eq45}) and Lemma \ref{lemma8},
 $R:\Omega\rightarrow \mathbb{R}$ is a finite random variable.


\subsection{Theorem}


The first main result of this section is stated below.
\begin{theorem}\label{maintheorem}
Let Assumptions 1-5 be fulfilled and suppose  $\vartheta \in (0,\min\{\theta,1-\alpha\})$.
Then there exists a finite random variable $C:\Omega\rightarrow [0,\infty)$ such that
for all $m\in\{0,1,...,M\}$ and  every $M,N\in \mathbb{N}$
\begin{equation*}
\|X_{m\Delta t}^{N}(\omega)-Y_{m}^{N,M}(\omega)\|_{V}
\leq C(\omega)(\Delta t)^{\vartheta},
\end{equation*}
 for all $\omega \in \Omega$,
 where  $X^{N}:[0,T]\times \Omega\rightarrow V$ is the unique adapted stochastic process with continuous sample paths,
 defined in Assumption 4, and  $Y_{m}^{N,M}:\Omega\rightarrow V$, for $m\in \{0,1,...,M\},$ and $N,M\in \mathbb{N}$, 
is given in (\ref{Teq1}). 
\end{theorem}
%

%
\begin{proof}
For the proof it is sufficient to prove the result for sufficiently small $|t_2-t_1|$.
From (\ref{Teqq2}) we have
\begin{equation}\label{Teq6}
\begin{split}
X_{m\Delta t}^{N}(\omega)&=\int_0^{m\Delta t}P_{N}S_{m\Delta t-s}F(X_s^{N}(\omega))ds+O_{m\Delta t}^{N}(\omega)\\
&=\sum_{k=0}^{m-1}\int_{k\Delta t}^{(k+1)\Delta t}P_{N}S_{m\Delta t-s}F(X_s^{N}(\omega))ds+O_{m\Delta t}^{N}(\omega),
\end{split}
\end{equation}
for every $m\in \{0,1,...,M \}$, and every $M \in \mathbb{N}$.

as an intermediate discretization, we
consider the mapping $Y_{m}^{N}:\Omega\rightarrow V,~m=1,2,...,M$  by
\begin{equation}\label{Teq7}
Y_{m}^{N}(\omega)=\sum_{k=0}^{m-1}\int_{k\Delta t}^{(k+1)\Delta t}P_{N} S_{m\Delta t-k\Delta t}F(X_{k\Delta t}^{N}(\omega))ds+O_{m\Delta t}^{N}(\omega).
\end{equation}
Our aim is to bound
$\|X_{m\Delta t}^{N}(\omega)-Y_{m}^{N,M}(\omega)\|_{V}$. 

We split this into two Lemmas.
First for an error  between $Y_{m}^{N}$ and the spectral Galerkin method.


\begin{lemma}
\label{lem:bou1}
 Under the assumptions of Theorem \ref{maintheorem}
there exists a finite random variable $C:\Omega\rightarrow [0,\infty)$ such that
for all $m\in\{0,1,...,M\}$ and  every $M,N\in \mathbb{N}$
\begin{equation}\label{Teq9}
\|X_{m\Delta t}^{N}(\omega)-Y_{m}^{N}(\omega)\|_{V} \leq C(\omega) (\Delta t)^\vartheta,
\end{equation}
for every $\omega \in \Omega$. 
\end{lemma}
Secondly for the difference between $Y_{m}^{N}$ and the full discretization in time


  \begin{lemma}
\label{lem:bou2}
 Under the assumptions of Theorem \ref{maintheorem}
there exists a finite random variable $C:\Omega\rightarrow [0,\infty)$ such that
for all $m\in\{0,1,...,M\}$ and  every $M,N\in \mathbb{N}$
  \begin{equation}\label{error2}
\|Y_{m}^{N}(\omega)-Y_{m}^{N,M}(\omega)\|_{V} \leq C(\omega) \sum_{k=0}^{m-1} \|X_{k\Delta t}^{N}(\omega)-Y_{k}^{N,M}(\omega)\|_{V}.
\end{equation}
\end{lemma}


\subsection{Proof of Lemma \ref{lem:bou1}}


For estimating the first error term stated in (\ref{Teq9}) we have
\begin{equation}\label{Teq10}
\begin{split}
X_{m\Delta t}^{N}(\omega)-Y_{m}^{N}(\omega)=&
\sum_{k=0}^{m-2}\int_{k\Delta t}^{(k+1)\Delta t}P_{N}S_{m\Delta t-s}F(X_s^{N}(\omega))ds
\\&
-\sum_{k=0}^{m-2}\int_{k\Delta t}^{(k+1)\Delta t}P_{N}S_{m\Delta t-k\Delta t}F(X_{k\Delta t}^{N}(\omega))ds
\\&
+\int_{(m-1)\Delta t}^{m\Delta t}P_{N}S_{m\Delta t-s}F(X_s^{N}(\omega))ds 
\\&
-\int_{(m-1)\Delta t}^{m\Delta t}P_{N}S_{\Delta t}F(X_{k\Delta t}^{N}(\omega))ds.
\end{split}
\end{equation}
At first we obtain the bound for the last two integrals in (\ref{Teq10}). 
For the first one, we get
\begin{equation*}\label{Teq11}
\begin{split}
\Big\|\int_{(m-1)\Delta t}^{m\Delta t}P_{N}S_{m\Delta t-s}F(X_s^{N}(\omega))ds\Big\|_{V} &\leq \int_{(m-1)\Delta t}^{m\Delta t}\|P_{N}S_{m\Delta t-s}\|_{L(W,V)}\cdot\|F(X_s^{N}(\omega))\|_{W}ds\\
&\leq \int_{(m-1)\Delta t}^{m\Delta t} (m\Delta t-s)^{-\alpha}ds\sup_{0\leq s\leq t}\|F(X_s^{N}(\omega))\|_{W}\\
&\leq  C R(\omega) (\Delta t)^{1-\alpha}.
\end{split}
\end{equation*}
For the second term we obtain similarly
\begin{equation*}\label{Teq12}
\begin{split}
\Big\|\int_{(m-1)\Delta t}^{m\Delta t}P_{N}S_{\Delta t}F(X_{k\Delta t}^{N}(\omega))ds\Big\|_{V}
&\leq \sup_{0\leq s\leq t}\|F(X_s^{N}(\omega))\|_{W}\int_{(m-1)\Delta t}^{m\Delta t} (\Delta t)^{-\alpha}ds\\
&\leq R(\omega) (\Delta t)^{1-\alpha}.
\end{split}
\end{equation*}
Therefore we have
\begin{equation}\label{Teq100}
\begin{split}
\|X_{m\Delta t}^{N}(\omega)-Y_{m}^{N}(\omega)\|_{V}\leq &
\Big\|\sum_{k=0}^{m-2}\int_{k\Delta t}^{(k+1)\Delta t}P_{N}S_{m\Delta t-s}F(X_s^{N}(\omega))ds
\\&
-\sum_{k=0}^{m-2}\int_{k\Delta t}^{(k+1)\Delta t}P_{N}S_{m\Delta t-k\Delta t}F(X_{k\Delta t}^{N}(\omega))ds\Big\|_{V}
\\&
+R(\omega)(\Delta t)^{1-\alpha}.
\end{split}
\end{equation}
Now we insert the OU-process. Define 
\[Z_{s,k\Delta t}^N(\omega)=O_s^{N}(\omega)-O_{k\Delta t}^{N}(\omega)
\]
Thus for every $m\in \{0,1,...,M\}$ we have,
\begin{equation}\label{Serror}
\begin{split}
&\Big\|X_{m\Delta t}^{N}(\omega)-Y_{m}^{N}(\omega)\Big\|_{V}\\
&\leq \Big\|\sum_{k=0}^{m-2}\int_{k\Delta t}^{(k+1)\Delta t}P_{N}S_{m\Delta t-s}\Big(F(X_s^{N}(\omega))-F\big(X_{k\Delta t}^{N}(\omega)+Z_{s,k\Delta t}^N(\omega)\big)\Big)ds\Big\|_{V}\\
&+\Big\|\sum_{k=0}^{m-2}\int_{k\Delta t}^{(k+1)\Delta t}P_{N}S_{m\Delta t-s}\Big(F(X_{k\Delta t}^{N}(\omega)+Z_{s,k\Delta t}^N(\omega))-F(X_{k\Delta t}^{N}(\omega))\Big)ds\Big\|_{V}\\
&+\Big\|\sum_{k=0}^{m-2}\int_{k\Delta t}^{(k+1)\Delta t}\Big(P_{N}S_{m\Delta t-s}-P_{N}S_{m\Delta t-k\Delta t}\Big)F(X_{k\Delta t}^{N}(\omega))ds\Big\|_{V}\\
&+R(\omega) (\Delta t)^{1-\alpha}.\\
\end{split}
\end{equation}
For the first term in (\ref{Serror}) by using (\ref{drift condi}) and Lemma \ref{lemma8} 
we conclude
\begin{equation}\label{Teq15}
\begin{split}
& \Big\|\sum_{k=0}^{m-2}\int_{k\Delta t}^{(k+1)\Delta t}P_{N}S_{m\Delta t-s}\Big(F(X_s^{N}(\omega))-F\big(X_{k\Delta t}^{N}(\omega)+Z_{s,k\Delta t}^N(\omega)\big)\Big)ds\Big\|_{V}\\&
  \leq L\sum_{k=0}^{m-2}
\int_{k\Delta t}^{(k+1)\Delta t}(m\Delta t-s)^{-\alpha}
\Big\|X_s^{N}(\omega)-(X_{k\Delta t}^{N}(\omega)+Z_{s,k\Delta t}^N(\omega))\Big\|_{V}\\
&\qquad\qquad\qquad\qquad\qquad\cdot\Big(1+\|X_s^{N}(\omega)\|_{V}^{p}+\|X_{k\Delta t}^{N}(\omega)+Z_{s,k\Delta t}^N(\omega)\|_{V}^{p}\Big)ds\\&
 \leq C(\omega)\sum_{k=0}^{m-2}\int_{k\Delta t}^{(k+1)\Delta t}(m\Delta t-s)^{-\alpha}(s-k\Delta t)^{\vartheta}\Big(1+2R^{p}(\omega)+(s-k\Delta t)^{p\vartheta}\Big)ds\\&
 \leq C(\omega) (\Delta t)^{\vartheta}.
\end{split}
\end{equation}  
For the second term in (\ref{Serror}) by (\ref{drift condi}) we derive
\begin{equation}\label{Teq16}
\begin{split}
&\Big\|\sum_{k=0}^{m-2}\int_{k\Delta t}^{(k+1)\Delta t}P_{N}S_{m\Delta t-s}\Big(F\big(X_{k\Delta t}^{N}(\omega)+Z_{s,k\Delta t}^N(\omega)\big)-F(X_{k\Delta t}^{N}(\omega))\Big)ds\Big\|_{V}\\
&\leq L\sum_{k=0}^{m-2}\int_{k\Delta t}^{(k+1)\Delta t}\Big\|P_{N}S_{m\Delta t-s}\Big\|_{L(W,V)}\Big\|X_{k\Delta t}^{N}(\omega)+Z_{s,k\Delta t}^N(\omega)-X_{k\Delta t}^{N}(\omega)\Big\|_{V}\\
&\qquad \qquad \qquad  \qquad \cdot 
\Big(1+\Big\|X_{k\Delta t}^{N}(\omega)+Z_{s,k\Delta t}^N(\omega)\Big\|_{V}^{p}+\|X_{k\Delta t}^{N}(\omega)\|_{V}^{p}\Big)
\\
&\leq C(\omega) \sum_{k=0}^{m-2}\int_{k\Delta t}^{(k+1)\Delta t}(m\Delta t-s)^{-\alpha}(s-k\Delta t)^{\vartheta}\Big(1+2R^{p}(\omega)+(s-k\Delta t)^{p\vartheta}\Big)\\& \leq
C(\omega)(\Delta t)^{\vartheta}.
\end{split}
\end{equation}

Finally, for the third term in (\ref{Serror}) we drive
\begin{equation}\label{Teq18}
\begin{split}
&\Big\|\sum_{k=0}^{m-2}\int_{k\Delta t}^{(k+1)\Delta t}\left(P_{N}S_{m\Delta t-s}-P_{N}S_{m\Delta t-k\Delta t}\right)F(X_{k\Delta t}^{N}(\omega))ds\Big\|_{V}\\
& \leq \sum_{k=0}^{m-2}\int_{k\Delta t}^{(k+1)\Delta t}\Big\|P_{N}S_{m\Delta t-k\Delta t}(S_{k\Delta t-s}-I)F(X_{k\Delta t}^{N}(\omega))\Big\|_{V}ds
\\& \leq
\sum_{k=0}^{m-2}\int_{k\Delta t}^{(k+1)\Delta t}\|P_{N}S_{\frac{m\Delta t-k\Delta t}{2}}\|_{L(W,V)}
\cdot\|A^{\theta}S_{\frac{m\Delta t-k\Delta t}{2}}\|_{L(V,V)}
\\& \qquad \qquad \qquad  \qquad \qquad  
\cdot\|A^{-\theta}(S_{k\Delta t-s}-I)\|_{L(V,V)}
\cdot\|F(X_{k\Delta t}^{N}(\omega))\|_{W}ds\\
&\leq C(\omega) \sum_{k=0}^{m-2}\int_{k\Delta t}^{(k+1)\Delta t}(m\Delta t-k\Delta t)^{-\alpha-\theta}(k\Delta t-s)^{\theta}\|F(X_{k\Delta
t}^{N}(\omega))\|_{W}ds\\& \leq
C(\omega)(\Delta t)^{\theta}\leq C(\omega)(\Delta t)^{\vartheta},
\end{split}
\end{equation}
where we used (\ref{e:SGtime}) from the assumption on the semigroup.

Hence, from (\ref{Teq15}), (\ref{Teq16}) and (\ref{Teq18}) we get
 \begin{equation}\label{Teq19}
\begin{split}
&\|X_{m\Delta t}^{N}(\omega)-Y_{m}^{N}(\omega)\|_{V} \leq C(\omega) (\Delta t)^{\vartheta}.
\end{split}
\end{equation}


\subsection{Proof of Lemma \ref{lem:bou2}}


Now, for the second error term in (\ref{error2}) because $Y_{m}^{N,M}:\Omega\rightarrow V$ satisfies
\begin{equation}\label{Teq21}
Y_{m}^{N,M}(\omega)=\sum_{k=0}^{m-1}\int_{k\Delta t}^{(k+1)\Delta t}P_{N} S_{m\Delta t-k\Delta t}F(Y_{k}^{N,M}(\omega))ds+P_{N}O_{m\Delta t}(\omega).
\end{equation}
and by using (\ref{Numeric data}), we can estimate
 \begin{equation}\label{Teq22}
\begin{split}
&\|Y_{m}^{N}(\omega)-Y_{m}^{N,M}(\omega)\|_{V}=\Big\|\sum_{k=0}^{m-1}\int_{k\Delta t}^{(k+1)\Delta t}P_{N}S_{m\Delta t-k\Delta t}(F(X_{k\Delta t}^{N}(\omega))-F(Y_{k}^{N,M}(\omega)))ds\Big\|_{V}\\
&\leq L\sum_{k=0}^{m-1}\int_{k\Delta t}^{(k+1)\Delta t}(m\Delta t-k\Delta t)^{-\alpha}\Big\|X_{k\Delta t}^{N}(\omega)-Y_{k}^{N,M}(\omega)\Big\|_{V}\Big(1+\Big\|X_{k\Delta t}^{N}(\omega)\Big\|^{p}_{V}+\Big\|Y_{k}^{N,M}(\omega)\Big\|^{p}_{V}\Big)ds\\
&\leq C(\omega)\sum_{k=0}^{m-1}\Delta t(m\Delta t-k\Delta t)^{-\alpha}\Big\|X_{k\Delta t}^{N}(\omega)-Y_{k}^{N,M}(\omega)\Big\|_{V}.
\end{split}
\end{equation}
Therefore we have
\begin{equation}
\|Y_{m}^{N}(\omega)-Y_{m}^{N,M}(\omega)\|_{V} \leq
 C(\omega) \sum_{k=0}^{m-1} \|X_{k\Delta t}^{N}(\omega)-Y_{k}^{N,M}(\omega)\|_{V}.
\end{equation}
\end{proof}
Now from Lemma \ref{lem:bou1} with Lemma \ref{lem:bou2}, we get
\begin{equation}\label{Teq24}
\begin{split}
\|X_{m\Delta t}^{N}&(\omega)-Y_{m}^{N,M}(\omega)\|_{V} 
\leq C\Big(R(\omega), \theta,T,L\Big)(\Delta t)^{\vartheta}\\
&+ C(\omega) \sum_{k=0}^{m-1} \|X_{k\Delta t}^{N}(\omega)-Y_{k}^{N,M}(\omega)\|_{V}.
\end{split}
\end{equation}
By the discrete Gronwall Lemma we finally conclude
\begin{equation*}\label{Teq27}
\|X_{m\Delta t}^{N}(\omega)-Y_{m}^{N,M}(\omega)\|_{V}
\leq C \Big(R(\omega),\theta,T,L\Big)(\Delta t)^{\vartheta}.
\end{equation*}


\subsection{Main results -- Full Discretization}


Combining Theorem 
\ref{maintheorem} for the time discretization and Theorem \ref{theoremBJ}
for the spatial discretization, yields the following result 
on the full discretization

\begin{theorem}\label{maintheorem2}
Let Assumptions 1-5 be true.
Let $X:[0,T]\times \Omega\rightarrow V$ be the solution of the SPDE (\ref{eqBJ1})
and  $Y_{m}^{N,M}:\Omega\rightarrow V$, $m\in \{0,1,...,M\}, M,N\in \mathbb{N}$ the numerical 
solution given by (\ref{Teq1}). Fix $\vartheta \in (0,\min\{\theta,1-\alpha\})$, then there exists a finite random variable $C:\Omega\rightarrow [0,\infty)$
such that
\begin{equation}\label{Teq31}
\|X_{m\Delta t}(\omega)-Y_{m}^{N,M}(\omega)\|_{V}\leq C(\omega)\left(N^{-\gamma}+(\Delta t)^{\vartheta}\right)
\end{equation}
for all $m\in\{0,1,...,M\}$ and  every $M,N\in \mathbb{N}$.
\end{theorem}

For simplicity of presentation we supposed in Theorem \ref{maintheorem2} both the 
full discretization (\ref{Numeric data}) 
and the Galerkin approximation (\ref{Teq1}) 
to be uniformly bounded. 

Following the proofs, it is easy to verify that it is sufficient to assume only 
one of those assumptions. Let us comment in more detail on the extension of 
the approximation result in that case.
Let us focus on the case where the uniform bound (\ref{Numeric data}) for the full discretization  
is not satisfied.

First it is easy to verify that the following minor modification of our main result is true.
Its proof follows directly, by observing, 
that the proof of the main theorem never uses the supremum over $M$ or $N$.

\begin{theorem}\label{maintheorem_modify}
Let Assumptions 1-4 be true. Fix $\vartheta \in (0,\min\{\theta,1-\alpha\})$ and fix a non-negative random constant $K(\omega)$.
Let $X:[0,T]\times \Omega\rightarrow V$ be the solution of the SPDE (\ref{eqBJ1})
and  $Y_{m}^{N,M}:\Omega\rightarrow V$, $m\in \{0,1,...,M\}, M,N\in \mathbb{N}$ the numerical 
solution given by (\ref{Teq1}). 

Then there exists a finite random variable $C:\Omega\rightarrow [0,\infty)$,
depending on $K$, but independent of $M$ and $N$,
such that the following is true:

If for one choice of $N,M\in\mathbb{N}$ we have 
\begin{equation}
\sup_{0\leq m\leq M} \|Y_{m}^{N,M}\|_{V}\leq K(\omega)
\end{equation}
then 
\begin{equation}
\|X_{m\Delta t}(\omega)-Y_{m}^{N,M}(\omega)\|_{V}\leq C(\omega)\left(N^{-\gamma}+(\Delta t)^{\vartheta}\right)
\end{equation}
for all $m\in\{0,1,...,M\}$.
\end{theorem}

To proceed, note that in the proofs
we can always bound  every occurrence of 
$\|Y_{m}^{N,M}\|^p_{V}$ by the bounded
$\|X_{t}^{N}\|_{V}$ and the error
\[e^{N,M}_t=\sup_{m\Delta t \leq t} \|Y_{m}^{N,M} - X_{m \Delta t}^{N}\|_{V}\;.
\]
I.e., we use for $m \leq t/\Delta t$
\[\|Y_{m}^{N,M}\|^p_{V} \leq C_p \|X_{m\Delta t}^{N}\|_{V}^p + C_p (e^{N,M}_t)^p\;.
\] 
If we now assume a-priori that $e^{N,M}_t \leq 1$, which is easily true, for sufficiently small $t>0$, 
then we can proceed completely analogous as in the proofs of Theorem \ref{maintheorem2}. 

By Theorem \ref{maintheorem_modify} this implies now that for the error, probably with a different $C(\omega)$ that
 \[e^{N,M}_t \leq C(\omega) (\Delta t)^\vartheta \;. 
 \]
 As the right hand-side is independent of $M$ and $N$,
we can a-posteriori conclude, that as long as $C(\omega) (\Delta t)^\vartheta \leq 1$
our initial guess on $e^{N,M}$ was true, and
we finally derive the following theorem:

\begin{theorem}\label{maintheorem_modify2}
Let Assumptions 1-4 be true. 
Let $X:[0,T]\times \Omega\rightarrow V$ be the solution of the SPDE (\ref{eqBJ1})
and  $Y_{m}^{N,M}:\Omega\rightarrow V$, $m\in \{0,1,...,M\}, M,N\in \mathbb{N}$ the numerical 
solution given by (\ref{Teq1}). 

Then there exists a finite random variable $C:\Omega\rightarrow [0,\infty)$
such that the error estimate
(\ref{Teq31}) holds provided $0<\Delta t< C(\omega)^{-\vartheta}$. 
\end{theorem}

The case when the uniform bound  (\ref{e:Galbou}) on the spectral Galerkin approximation fails, 
is verified in a similar  way, by bounding in the whole proof  $\|X_{t}^{N}\|_{V}^p$ by $\|Y_{m}^{N,M}\|^p_{V}$ and $e^{N,M}_t$.


\section{Numerical results}


In this section we consider examples for the numerical solution of stochastic equation 
by the method given in (\ref{Teq1}).
Let $ V = W = C([0,\pi])$ be the $\mathbb{R}$-Banach space
of continuous functions from $[0,\pi]$ to $\mathbb{R}$ equipped with the norm
$\|v\|_{V}=\|v\|_{W}:= \sup_{x\in [0,\pi]} |v(x)|$
for every $v\in V=W$ , where $|.|$ is the absolute value of a real number. Moreover,
consider as orthonormal $L^2$-basis the smooth eigenfunctions 
 \begin{equation*}
e_{k}:[0,\pi]\rightarrow \mathbb{R},~~~ e_{k}(x)=\sqrt{2/\pi}\sin(kx),
\qquad\text{for every } x\in (0,\pi).
\end{equation*}
Denote the Laplacian with Dirichlet boundary conditions on $[0,\pi]$ by $A$,
such that $Ae_k = - k^2e_k$. Moreover, define the operators  $P_N$ as the 
$L^2$-orthogonal projections onto the span of the first $N$ eigenfunctions $e_k$.

We define the mapping $S:[0,T]\rightarrow L(V)$  by 
\begin{equation}
(S_{t})v(x)=\sum_{i\in \mathbb{N}}e^{-\lambda_{i}t}\int e_{i}(s) v(s)ds\cdot e_{i}(x),
\end{equation}
where $\lambda_{i}=-i^{2}$.
It is well known that $A$ generates the analytic semigroup $(S_t)_{t\geq0}$ on $V$.
See \cite{Ref2013}.
From  Lemma 4.1 in \cite{Ref1} and Lemma 1 in \cite{Ref1362}
we recall that (\ref{S1}) is satisfied 
for $\gamma \in (0,\frac{3}{2})$ and $\alpha \in (\frac{1}{4}+\frac{\gamma}{2},1)$. Moreover, 
from \cite{Ref2013} we know that  (\ref{e:SGtime}) 
is satisfied for any $\theta\in (0,1)$.\\

Assume that the OU-process  $O:[0,T]\times \Omega\rightarrow V$ is 
as defined in the example in Section \ref{suse:OU}. Therefore $O$ satisfies    
Assumption 3, for all  $\theta \in (0,\min\{\frac{1}{2},\frac{\rho}{2}\})$ and $\gamma \in (0,\rho)$.
The covariance operator $Q$ is given 
as a convolution operator
\begin{equation}\label{MathQ}
\langle Q e_{k},e_{l} \rangle =\int_0^{\pi}\int_0^{\pi}e_{k}(x)e_{l}(y)q(x-y)dy dx.
\end{equation}
We obtain our numerical result with two kernels
\begin{equation}\label{cov1exam1}
q_1(x-y)=\frac1h\max\{0,1- \frac1{h^2} |x-y|\}
\end{equation}
and
\begin{equation}\label{cov2exam1}
q_2(x-y)=\max\{0,1 - \frac1h |x-y|\}.
\end{equation}

In Figures \ref{CovMatrix1}, \ref{CovMatrix2} we  plotted the Covariance Matrix for $h=0.1$ and $h=0.01$ with kernel (\ref{cov1exam1}) and kernel (\ref{cov2exam1}).

By some numerical calculations 
we can show that the condition on $Q$ from (\ref{MatQ}) 
is satisfied for any $\rho\in (0,\frac{1}{2})$,
as we can calculate explicitly 
the Fourier-coefficients of $(x,y)\mapsto q(x-y)$ 
and check for summability.

For simplicity fix the smooth deterministic initial condition  
\[
 \xi(x)=\frac{\sin x}{\sqrt{2}}+\frac{3\sqrt{2}}{5}\sin(3x), \quad\text{ for all }x\in [0,\pi].
\]

Now we consider two types of nonlinearity, globally Lipschitz and locally Lipschitz, as given by the following examples.
\\
\textbf{Example 1}
  Consider for the nonlinearity
the Nemytskii operator $F:V\rightarrow V$ given by $(F(v))(x) = f(v(x))$ for
every $x \in [0,\pi]$ and every $v \in V$, 
where $f: \mathbb{R}\rightarrow \mathbb{R}$ 
is given by 
\begin{equation}
f(y)=5\cdot\frac{1-y}{1+y^{2}}.
\end{equation}
This generates a globally Lipschitz nonlinearity.
Thus Assumption 2 is true.\\
The stochastic equation (\ref{intr1}) now reads as
\begin{equation}\label{eq4exam1}
dX_{t}=\Big[\frac{\partial ^{2}}{\partial x^{2}}X_{t}+5\frac{1-X_{t}}{1+X_{t}^{2}}\Big]dt+dW_{t},
~~X_0(x)=\frac{\sin x}{\sqrt{2}}+\frac{3\sqrt{2}}{5}\sin(3x),
\end{equation}
with Dirichlet boundary conditions $X_{t}(0)=X_{t}(\pi)=0$ 
for $t\in [0,1] $.

The finite dimensional SDE (\ref{S4}) reduces to
\begin{equation} 
dX_{t}^{N}
=\Big[\frac{\partial ^{2}}{\partial x^{2}}X_{t}^{N}
+5P_{N}\frac{1-X_{t}^{N}}{1+(X_{t}^{N})^{2}}\Big]dt
+d P_{N}W_{t},~~X_0^{N}(x)=\frac{\sin x}{\sqrt{2}}+\frac{3\sqrt{2}}{5}\sin(3x),
\end{equation}
with $X_{t}^{N}(0)=X_{t}^{N}(\pi)=0$ for $t\in [0,1]$ and $x\in [0,\pi]$, 
and all $N \in \mathbb{N}.$

Now in our simple example we can verify rigorously that the 
 numerical data is uniformly bounded. 
We derive
\begin{equation} 
\begin{split}
\|X_{t}^{N}(\omega)\|_{V}=&\Big\|\int_0^{t}P_{N}S_{t-s}F(X_s^{N}(\omega))ds+P_{N}(O_{t}(\omega))\Big\|_{V}\\
& \leq \int_0^{t}\|P_{N}S_{t-s}\|_{L(V,V)}\|F(X_s^{N}(\omega))\|_{V}ds+\|P_{N}(O_{t}(\omega))\|_{V}\\
& \leq C \int_0^{t}\|P_{N}S_{t-s}\|_{L(V,V)}(1+\|X_s^{N}(\omega)\|_{V})ds+C,
\end{split}
\end{equation}
from Gronwall inequality we conclude
\begin{equation}
\sup_{N\in \mathbb{N}}\sup_{0\leq s \leq T}\|X_{s}^{N}(\omega)\|_{V} <\infty. 
\end{equation}
In a similar way we conclude
\begin{equation}
\sup_{N,M\in \mathbb{N}}\sup_{0\leq m \leq M}\|Y_{m}^{N,M}(\omega)\|_{V} <\infty. 
\end{equation}
Theorem \ref{maintheorem2} by using this fact that here $\theta \in (0,\min\{\frac{1}{2},\frac{\rho}{2}\})$, yields the existence of a unique solution $X:[0,\pi]\times \Omega \rightarrow C^{0}([0,\pi])$ of the SPDE (\ref{eq4exam1})
such that
\begin{equation}
\sup_{0\leq x \leq \pi}|X_{m\Delta t}(\omega,x)-Y_{m}^{N,M}(\omega,x)|\leq C(\omega)\left(N^{-\gamma}+(\Delta t)^{\vartheta}\right)
\end{equation}
for $m=1,...,M,~~ M=\frac{1}{\Delta t}$, such that $\gamma \in (0,\frac{1}{2})$, $\vartheta \in(0,\frac{1}{4})$.\\

\textbf{Example 2}
  Consider for the nonlinearity
the Nemytskii operator $F:V\rightarrow V$ given by $(F(v))(x) = f(v(x))$ for
every $x \in [0,\pi]$ and every $v \in V$, 
where $f:\mathbb{R}\rightarrow \mathbb{R}$ 
is given by 
\begin{equation}
f(y)=-y^{3}.
\end{equation}
This generates a locally Lipschitz nonlinearity which satisfies Assumption \ref{ass:nonlin}.
The stochastic equation (\ref{intr1}) now reads as
\begin{equation}\label{eq4exam2}
dX_{t}=\Big[\frac{\partial ^{2}}{\partial x^{2}}X_{t}-X_{t}^{3}\Big]dt+dW_{t},
~~X_0(x)=\frac{\sin x}{\sqrt{2}}+\frac{3\sqrt{2}}{5}\sin(3x),
\end{equation}
with Dirichlet boundary conditions $X_{t}(0)=X_{t}(\pi)=0$ 
for $t\in [0,1] $.

The finite dimensional SDE (\ref{S4}) reduces to
\begin{equation}\label{eq5exam1}
dX_{t}^{N}
=\Big[\frac{\partial ^{2}}{\partial x^{2}}X_{t}^{N}
-P_{N}(X_{t}^{N})^{3}\Big]dt
+d P_{N}W_{t},~~X_0^{N}(x)=\frac{\sin x}{\sqrt{2}}+\frac{3\sqrt{2}}{5}\sin(3x),
\end{equation}
with $X_{t}^{N}(0)=X_{t}^{N}(\pi)=0$ for $t\in [0,1]$ and $x\in [0,\pi]$, 
and all $N \in \mathbb{N}.$

Now by
using Theorem \ref{maintheorem_modify2} it remains to verify (\ref{e:Galbou})
from Assumption 4. This is straightforward by using first the estimates in $L^2$, 
and we sketch only the main ideas here.

Define $y^N_t=X_t^N-P_{N}O(t)$. Thus
\[
\partial_t y^N_t = \frac{\partial ^{2}}{\partial x^{2}}y^N_t -(y^N_t +P_{N}O_t)^{3}\;.
\]
Hence,
\begin{equation} 
\tfrac12\partial_t\|y_{t}^{N}\|^2_{L^2} 
= -\| \partial_x  y_{t}^{N}\|^2 - \tfrac12 \|y_{t}^{N}\|^4_{L^4} + C \|P_{N}O_t\|^4_{L^4}\;.
\end{equation}
This gives a random bound in $L^2([0,T], H^1) \cap L^\infty([0,T], L^2)\cap L^4([0,T], L^4)$.
Using Agmon inequality yields
\[   \|y^N_t\|_{L^4([0,T],V)}^2 \leq C  \|y^N_t\|_{L^2([0,T],H^1)} \cdot \|y^N_t\|_{L^\infty([0,T],L^2)}\;.
\]
This is sufficient to verify the bound in $V$ from the mild formulation, as 
\[
 \|y^N_t\|_{V} \leq  C \int_0^{t}\|P_{N}S_{t-s}\|_{L(V,V)}\cdot  \|y^N_s+P_N O_s\|_{V}^3 ds\;.
 \]
 
Now from our main results 
for the  unique solution $X:[0,\pi]\times \Omega \rightarrow C^{0}([0,\pi])$ of the SPDE (\ref{eq4exam2})
we obtain for sufficiently small $\Delta t$
\begin{equation}
\sup_{0\leq x \leq \pi}|X_{m\Delta t}(\omega,x)-Y_{m}^{N,M}(\omega,x)|\leq C(\omega)\left(N^{-\gamma}+(\Delta t)^{\vartheta}\right)
\end{equation}
for $m=1,...,M,~~ M=\frac{1}{\Delta t}$, such that $\gamma \in (0,\frac{1}{2})$, $\vartheta \in(0,\frac{1}{4})$.\\
Let us now explain briefly how we implement our numerical results. 
The main part is generating the Brownian motions $X=(X_1,X_2,\cdots,X_N)$ that are correlated such that $X\sim \textit{N}(0,\Sigma)$, which $Cov(X_i,X_j)=\Sigma_{ij}$. 
For this 
assume $C$ is a $n\times m$ Matrix and let $Z=(Z_{1},\cdots,Z_{N})^{T}$, 
with $Z_{i}\sim \textit{N}(0,1)$, for $i=1,\cdots,N.$ 
Then obviously $C^{T}Z\sim N(0,C^{T}C)$. 
Therefore our aim clearly reduces to finding $C$ such that $C^{T}C=\Sigma$,
which can for instance be achieved by Cholesky. 
By using $\Delta t=\frac{T}{N^{2}}$,  
the solutions $X_{t}^{N}(\omega,x)$ of the finite dimensional SODEs (\ref{eq5exam1}) converge uniformly in
 $t\in [0,1]$ and $x\in [0,\pi]$ to the solution $X_{t}(\omega,x)$ 
of the stochastic evolution equation (\ref{eq4exam2}) with the rate $\frac{1}{2}$, 
as $N$ goes to infinity for all $\omega \in \Omega.$ 
In Figure \ref{Figpatherr} the path-wise approximation error
 \begin{equation}\label{patherr}
\sup_{0\leq x \leq \pi}\sup_{0\leq m \leq M}|X_{m\Delta t}(\omega,x)-Y_{m}^{N,M}(\omega,x)|\
\end{equation}
is plotted against $N$, for $N \in \{16,32,\cdots,256\}$.
As a replacement for the true unknown solution, 
we use a numerical approximation for $N$ sufficiently large.  

Figure \ref{Figpatherr} confirms that, as we expected from Theorem \ref{maintheorem_modify2}, 
the order of convergence is $\frac{1}{2}.$ Obviously, 
these are only two examples, but all out 
of a few hundred calculated examples behave similarly. 
Even their mean seem to behave with 
the same order of the error. Nevertheless, we did not calculate sufficiently 
many realizations to estimate the mean satisfactory, 
nor did we proof in the general setting, that the mean converges. 

Finally, as an example in Figures \ref{Evolution}, 
 $X_{t}(\omega)$, are plotted for $t \in  [0,T]$ for $T\in \{\frac{3}{200},0.2,1\},$
 for $h=0.1$, with convolution operator (\ref{MathQ})
 with kernel (\ref{cov1exam1}) and (\ref{cov2exam1}).


 \begin{figure}
	\centering
	\hfill %
	\subfloat[]{\includegraphics[width=0.48\textwidth]{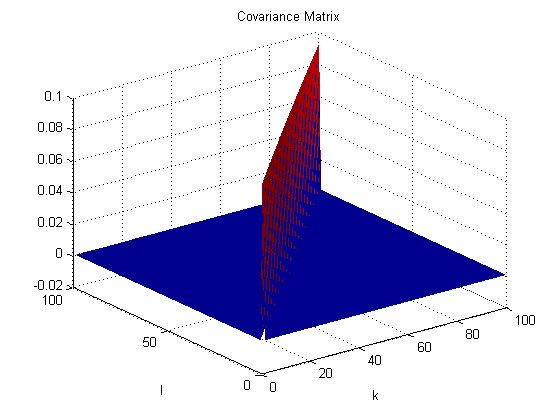}}
	\hfill %
	\subfloat[]{\includegraphics[width=0.48\textwidth]{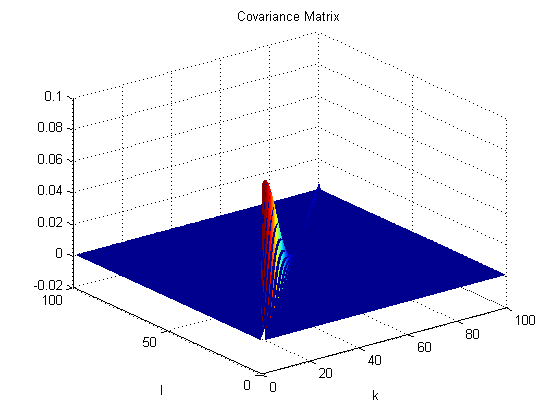}}
	\hfill %
	\caption[Covariance Matrix ]{Covariance Matrix $<Qe_{k},e_{l}>_{k,l}$ for $k,l \in \{1,2,\cdots,100\}$, for $h=0.1$ by (a) kernel (\ref{cov1exam1})  and (b) kernel (\ref{cov2exam1})}
	\label{CovMatrix1}
\end{figure}

\begin{figure}
	\centering
	\hfill %
	\subfloat[]{\includegraphics[width=0.48\textwidth]{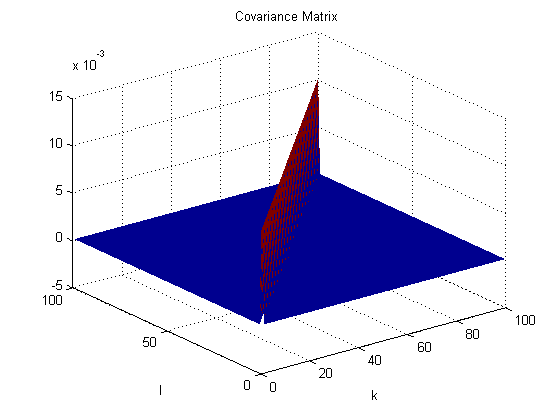}}
	\hfill %
	\subfloat[]{\includegraphics[width=0.48\textwidth]{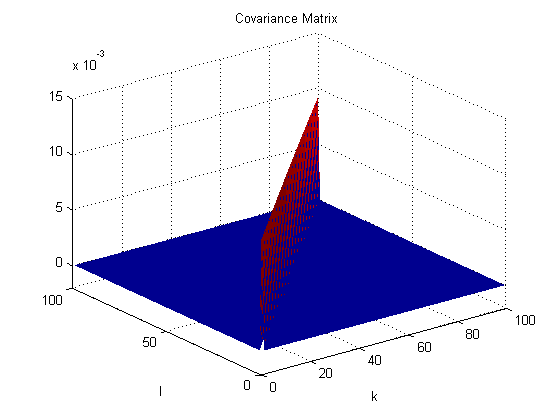}}
	\hfill %
	\caption[Covariance Matrix ]{Covariance Matrix $<Qe_{k},e_{l}>_{k,l}$ for $k,l \in \{1,2,\cdots,100\}$, for $h=0.01$  (a) kernel (\ref{cov1exam1})  and (b) kernel (\ref{cov2exam1})}
	\label{CovMatrix2}
\end{figure}

\begin{figure}
	\centering
	\hfill %
	\subfloat[]{\includegraphics[width=0.48\textwidth]{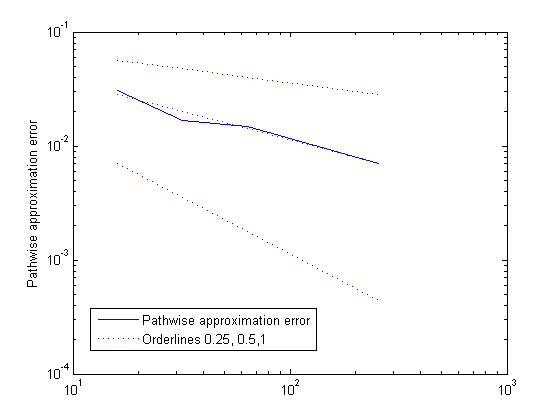}}
	\hfill %
	\subfloat[]{\includegraphics[width=0.48\textwidth]{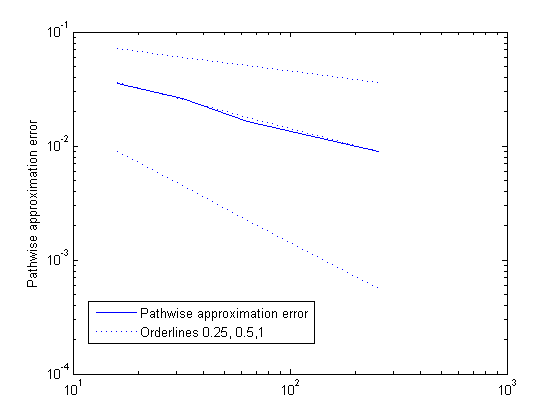}}
	\hfill %
	\caption[Pathwise approximation error]{Pathwise approximation error (\ref{patherr}) 
against $N$  for $N \in \{16,32,...,256\}$ with convolution operator with kernel (\ref{cov1exam1}) for (a) $h=0.1$  and (b) $h=0.01$, for one random $\omega \in \Omega$.} 
\label{Figpatherr}
\end{figure} 
\begin{figure}
	\centering
	\hfill %
	{\includegraphics[width=0.48\textwidth]{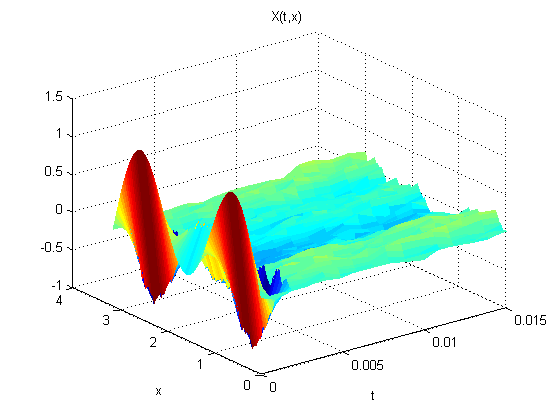}}
	\hfill %
	{\includegraphics[width=0.48\textwidth]{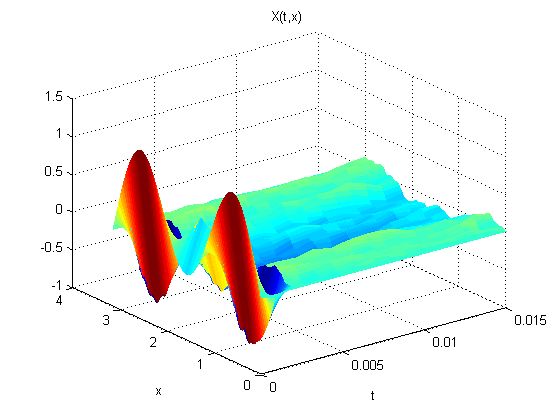}}
	\hfill %
	\end{figure}
\begin{figure}
	\centering
	\hfill %
	{\includegraphics[width=0.48\textwidth]{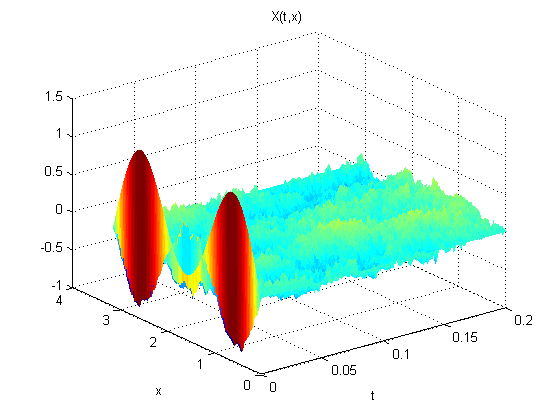}}
	\hfill %
	{\includegraphics[width=0.48\textwidth]{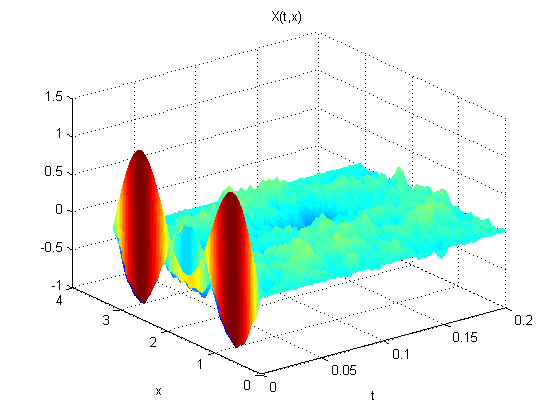}}
	\hfill %
	\end{figure}
\begin{figure}
	\centering
	\hfill %
	\subfloat[]{\includegraphics[width=0.48\textwidth]{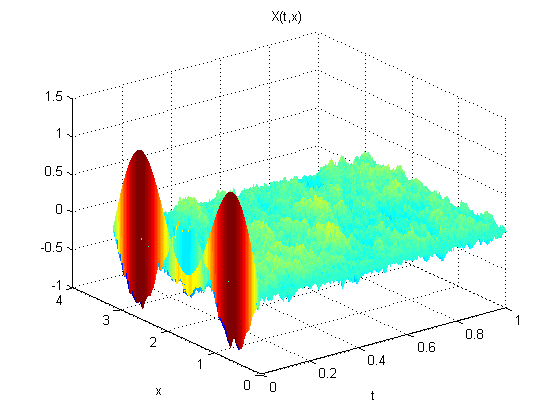}}
	\hfill %
	\subfloat[]{\includegraphics[width=0.48\textwidth]{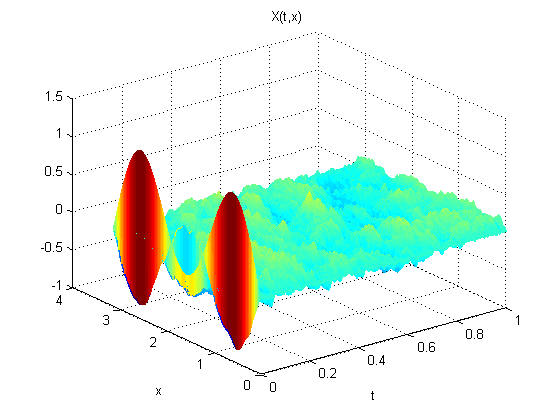}}
	\hfill %
	\caption[Stochastic Evolution equation]{ $X_{t}(\omega,x)$, $ x \in [0,\pi],~t\in (0,T)$ for $T\in \{3/200,0.2,1\}$, given by (\ref{eq4exam1}) for $h=0.1$ with the covariance operator by (a) kernel (\ref{cov1exam1})  and (b) kernel (\ref{cov2exam1}), for one random $\omega \in \Omega$.}
	\label{Evolution}
\end{figure}

\end{document}